\documentclass[10t,a4paper]{amsart}

\usepackage[english]{babel}
\usepackage[T1]{fontenc}
\usepackage[utf8]{inputenc}

\usepackage{amscd,amssymb,amsopn,amsmath,amsthm,graphics,amsfonts,enumerate,verbatim,calc}
\usepackage{amssymb, tikz}
\usepackage{mathtools}

\usepackage{hyperref}
\usepackage[capitalise]{cleveref}

\usepackage{setspace}
\usepackage{mathpazo}
\usepackage{color}
\usepackage{latexsym}
\usepackage{amsthm,amsfonts,amssymb,mathrsfs}
\usepackage{rotating}
\usepackage[leqno]{amsmath}
\usepackage{xspace}
\usepackage[all]{xy}
\usepackage{longtable}
\headsep=5 mm \numberwithin{equation}{section}
\newtheorem{theorem}{Theorem}[section]
\newtheorem{lemma}[theorem]{Lemma}

\newtheorem{proposition}[theorem]{Proposition}
\newtheorem{corollary}[theorem]{Corollary}

\newtheorem{problem}[theorem]{Problem}
\theoremstyle{definition}

\theoremstyle{remark}
\newtheorem{remark}[theorem]{Remark}
\newtheorem{fact}[theorem]{Fact}
\newtheorem{example}[theorem]{Example}
\newtheorem{observation}[theorem]{Observation}
\newtheorem{discussion}[theorem]{Discussion}
\newtheorem{question}[theorem]{Question}
\newtheorem{conjecture}[theorem]{Conjecture}

\newtheorem{acknowledgement}{Acknowledgement}

\DeclareMathOperator{\CI}{\textnormal{CI-dim}}

\newcommand{\UFD}{\operatorname{UFD}}
\newcommand{\cl}{\operatorname{Cl}}

\newcommand{\Pic}{\operatorname{Pic}}
\newcommand{\Spec}{\operatorname{Spec}}

\newcommand{\G}{G}\newcommand{\Tr}{\operatorname{Tr}}

\newcommand{\tors}{\operatorname{tors}}

\newcommand{\pd}{\operatorname{pd}}

\newcommand{\hh}{\operatorname{H}}

\newcommand{\Soc}{\operatorname{Soc}}

\newcommand{\rank}{\operatorname{rank}}
\newcommand{\Gdim}{\operatorname{Gdim}}
\newcommand{\Se}{\operatorname{S}}

\newcommand{\V}{\operatorname{Var}}
\newcommand{\Proj}{\operatorname{Proj}}
\newcommand{\id}{\operatorname{id}}
\newcommand{\Ext}{\operatorname{Ext}}

\newcommand{\R}{\operatorname{R}}
\newcommand{\Supp}{\operatorname{Supp}}

\newcommand{\Tor}{\operatorname{Tor}}

\newcommand{\Hom}{\operatorname{Hom}}

\newcommand{\End}{\operatorname{End}}

\newcommand{\depth}{\operatorname{depth}}

\newcommand{\coker}{\operatorname{coker}}

\newcommand{\lo}{\longrightarrow}
\newcommand{\fm}{\frak{m}}
\newcommand{\fp}{\frak{p}}

\newcommand{\fa}{\frak{a}}

\newcommand{\fn}{\frak{n}}

\newcommand{\PP}{\mathbb{P}}

\begin{document}

\author[]{Mohsen Asgharzadeh}

\address{}
\email{mohsenasgharzadeh@gmail.com}

\title[ ]
{A note on  $\UFD$ }


\begin{abstract} 
We investigate conditions under which height-one ideals are principal. As a representative case, let $R$ be a strongly normal, almost factorial, complete intersection local ring, and let $\fp$ be a prime ideal of height one. We show that if
$
\depth(R/\fp)\geq \dim R-2,
$
then $\fp$ is principal.
As an immediate application, we use elementary local cohomology techniques to reprove the celebrated Auslander--Buchsbaum theorem, thereby providing a streamlined approach to certain results of Dao and shedding new light on a problem of Samuel. As a further consequence, we prove that local rings of multiplicity at most three are hypersurfaces.
We also give an affirmative answer to a question of Braun concerning reflexive ideals of finite injective dimension. Finally, we show how the reflexive hull can be recovered from the ideal transform by a simple argument, thereby providing a simplified proof of a result of Hartshorne. Then, we compute
ideal transform of tensor product $D_{\fm}(M\otimes_RN)$ in terms of
$\Hom_R(M^{*},N)$.
\end{abstract}

\maketitle
\tableofcontents

\section{Introduction}

The unique factorization property of regular local rings was posed by Krull and resolved by Auslander–Buchsbaum \cite{AB}. Despite their close connection, there are essential differences between the notions of unique factorization domain ($\UFD$) and regularity. In this context, Samuel conjectured, and Grothendieck later proved the following:

\begin{theorem}[Grothendieck, 1961]\label{groth}
	Let $(R, \fm)$ be a local complete intersection domain. If $R_P$ is a $\UFD$ for all $P \in \Spec(R)$ with $\mathrm{ht}(P) \leq 3$, then $R$ is a $\UFD$.
\end{theorem}

For an alternative proof, see \cite{moh}. Samuel originally posed the question for hypersurfaces \cite{S}, which naturally leads to the following:

\begin{problem}
	Is the $\UFD$ condition in codimension $\leq 3$ essential?
\end{problem}

Let $k$ be a field of characteristic not equal to $2$, and let $f$ be a non-degenerate quadratic form in $S_n := k[X_1, \ldots, X_n]$. Samuel proved the following:

\begin{enumerate}[i)]
	\item Either $\cl(S_3/(f)) \cong \mathbb{Z}/2\mathbb{Z}$, or else $S_3/(f)$ is factorial. If $k$ is algebraically closed, then $\cl(S_3/(f)) \cong \mathbb{Z}/2\mathbb{Z}$.
	\item $\cl(S_4/(f))$ is either infinite cyclic or trivial. It is infinite cyclic when $k$ is algebraically closed.
\end{enumerate}

Here, $\cl(-)$ denotes the divisor class group. See \cite{f} for more background on this topic. These examples demonstrate that additional assumptions may be necessary. In this direction, we present an elementary proof of the following observation:

\begin{observation}[Dao]
	Let $(S, \fn)$ be an equicharacteristic or unramified regular local ring of dimension $4$. Let $R$ be such that $\widehat{R} \cong S/(f)$ for some $f \in \fn$. If $R$ has an isolated singularity and is almost factorial, then $R$ is a $\UFD$.
\end{observation}

The proof presented here, uses basic properties of local cohomology, which also allow us to reprove the Auslander–Buchsbaum theorem mentioned above (see \S3). 
Our next goal is to understand Samuel’s intriguing problem:

\begin{problem}[Samuel]\label{51}
	Let $R$ be a $d$-dimensional Cohen–Macaulay ring with isolated singularity and $d > 3$. When is $R$ a $\UFD$?
\end{problem}
We present results  both in positive side of Problem \ref{51} and its negative side. 
Our analysis yields the following application:

\begin{corollary}
	Let $(R, \fm)$ be a complete local ring of depth at least $3$ containing $\mathbb{Q}$, and suppose $R$ has an isolated singularity with $\dim R > 4$. If the multiplicity $e(R) \leq 3$, then $R$ is a hypersurface.
\end{corollary}

This extends a result of Huneke \cite{h}, who treated the case $e(R) < 3$.

Let $\Spec^1(R)$ denote the set of prime ideals of height one. Auslander and Buchsbaum essentially studied primes $\fp \in \Spec^1(R)$ with finite projective dimension and $\depth(R/\fp) \geq \dim R - 2$; see \cite[Corollary 2]{AB}. In \S5, we consider the Gorenstein analogue and show:

\begin{theorem}
	Let $(R, \fm)$ be a strongly normal, almost factorial complete intersection ring, and let $\fp \in \Spec^1(R)$. If $\depth(R/\fp) \geq \dim R - 2$, then $\fp$ is principal.
\end{theorem}

This result slightly extends a recent theorem of \v{C}esnavi\v{c}ius and Scholze \cite[Theorem 1.1.3]{cep}.

We use our local cohomology techniques to give an affirmative answer to the following question posed by Braun:

\begin{question}[Braun {\cite[Question 16]{B}}]\label{b}
	Let $(R, \fm)$ be a normal domain and $I \lhd R$ a reflexive ideal with $\id_R(I) < \infty$. Is $I$ isomorphic to a canonical module?
\end{question}
We note that the present topic is intimately connected with our study of principal ideals, representing, in a sense, the dual aspect of that theory. As a final application, we employ local cohomology methods to sharpen a result implicit in Hartshorne's work \cite[Proposition 1.6]{Ha} by establishing a direct relationship between the reflexive hull and the ideal transform:

\begin{fact}[After Hartshorne]
	Let \( M \) be locally free over \( \operatorname{Spec}(R) \setminus \{\mathfrak{m}\} \) and let \( R \) be a local ring with \( \operatorname{depth}(R) \ge 2 \). There is an exact sequence
	\(
	0 \to \Gamma_{\fm}(M) \to M \to M^{**} \to H^1_{\fm}(M) \to 0 .
	\)
	In particular,
	\[
	M^{**} = D_{\fm}(M) := \bigcup_{n=1}^{\infty} \Hom(\mathfrak{m}^n, M).
	\]
\end{fact}
Our  proof is based on local algebra, and computes ideal transform of tensor product $D_{\fm}(M\otimes_RN)\cong \Hom_R(M^{*},N)$ when $\depth(N)>1$.

\medskip
\section{Dao's 3-dimensional hypersurfaces}
 In this note $(R,\fm,k)$ is a commutative noetherian local ring,
and modules are finitely generated, otherwise specialized.
The notation $\pd_R(-)$  stands for the projective  dimension
of $(-)$. 

\begin{fact}\label{uf}
The ring $R$ is $\UFD$ iff any height one prime ideal is principal.
\end{fact}

The $i^{th}$ local cohomology of $(-)$ with
	respect to an ideal $\fa$ is defined by
	$\hh^{i}_{\fa}(M):={\varinjlim}_n\Ext^{i}_{R} (R/\fa^{n},
-)$.
	
\begin{lemma}\label{t1} (See \cite[Lemma 3.2]{ACS}). Assume $t$ is an integer such that $2\leq t\leq\depth(N)$ 	
	and $\Supp_R(\Ext^i_R(M,N))\subseteq\{\fm\}$ for all $i=1, 
\ldots, t-1$.   There is an injection $\Ext^{t-1}_R(M,N)\hookrightarrow\hh_{\fm}^{t}(\Hom_R(M,N))$.
\end{lemma}

Recall  from Auslander \cite{A2}
  that $M$  is called  tor-rigid provided that the vanishing of a single $\Tor^R_j(M,N)$ for some $N\in mod(R)$ and for some $j\geq 1$ forces the vanishing of $\Tor^R_i(M,N)$ for all $i\geq j$. 	

\begin{lemma}\label{Jo}(Jothilingam, \cite[Theorem]{Joth}). Assume $R$ is a local ring and let $M$ be tor-rigid. If $\Ext^n_R(M,M)=0$ for some $n\geq 1$, then $\pd_R(M)<n$.
\end{lemma}

Here, the ring	$R$ is called almost factorial if the class group of R is torsion, i.e. $R$ is $\mathbb{Q}$-factorial.

\begin{lemma}\label{d}(Dao, see \cite[Theorem 2.7(3)]{d}). Let $R$ be a local 3-dimensional hypersurface ring such that $\widehat{R} = S/(f)$ where $(S, \fn)$ is an equicharacteristic or unramified regular local ring and $f\in \fn$. If $R$ is almost factorial with isolated singularity, then every finitely generated $R$-module is tor-rigid.
\end{lemma}

	By $(G_i)$ (resp.   $(\R_i)$) we mean $R_\fp$ is Gorenstein (resp. regular) for all $\fp\in\Spec(R)$ of  height at most $i$. 
Recall that  a module $M$ satisfies $(\Se_i)$ if $\depth(M_{\fp})\geq \min\{i,\dim(M_\fp)\}$ for all $\fp\in\Spec(R)$.

\begin{theorem}
Let $(R,\fm)$ be a local 3-dimensional hypersurface ring such that $\widehat{R} = S/(f)$ where $(S, \fn)$ is an equicharacteristic or unramified regular local ring and $f\in \fn$. If $R$ is almost factorial with isolated singularity, then $R$ is $\UFD$.
\end{theorem}

\begin{proof}
	Let $\fp\in\Spec^1(R)$, i.e., a prime ideal  of height one. Since the ring
is	with isolated singularity, $\fp$ is locally principal over the punctured spectrum (see Fact \ref{uf}). By Serre, $(\Se_2)$+$(\R_1)$  characterizes normality. So, $R$ is normal. In fact, it is supernormal. Applying this along with the determinate trick, we observe that 
	
$$R\subseteq\Hom_R(\fp,\fp)\subseteq\overline{R}=R\quad(\ast)$$Here, $(-)^\ast$ means $\Hom_R(-,R)$. As $\fp=\fp^{\ast\ast}$ we know $\depth(\fp)=\depth(\hom(\fp^\ast,R))\geq 2$. In view of Lemma
	\ref{t1} we know there is an injection $$\Ext^{1}_R(\fp,\fp)\hookrightarrow\hh_{\fm}^{2}(\Hom_R(\fp,\fp))\stackrel{(\ast)}=\hh_{\fm}^{2}(R).$$ Since $\depth(R)=3$ we have $\hh_{\fm}^{2}(R)=0$. Consequently, 
	we deduce that $\Ext^{1}_R(\fp,\fp)=0$. According to Lemma \ref{d}, $\fp$ is tor-rigid. Let us apply Jothilingam's result to deduce that $\fp$ is free. As free ideals are principal, we observe that $\fp$ is principal. By Fact \ref{uf} $R$ is UFD.
\end{proof}
\medskip
\section{A new proof of regular rings are UFD}

In this section we apply some local cohomology and reprove a celebrated theorem of Auslander--Buchsbaum \cite{AB}.
We apply their strategy, but almost every thing is different from \cite{AB}.

\begin{lemma}\label{to}
Let $R$ be of depth at most two, and $\fp\in\Spec^1(R)$. If $\pd_R(R/\fp)< \infty$, then $\fp$ is principal.
\end{lemma}

\begin{proof}
	Suppose first that $\fp=\fm$. This says  $R$ is DVR, and the claim is clear. Now, we can assume $\fp\neq\fm$. Then $\depth(R/\fp)>0$. By Auslander-Buchsbaum formula,
	$$\depth(R)=\depth(R/\fp)+\pd_R(R/\fp).$$
	This gives
	$\pd_R(R/\fp)\leq1$, and so $\fp$ is free.
\end{proof}
\begin{proposition}\label{aus}
Let	$\fp\in\Spec^1(R)$. If $\pd_R(R/\fp)\leq 2$, then $\fp$ is principal.
	
\end{proposition}

\begin{proof}
	The proof is by induction on $d:=\dim(R)$.
By Lemma \ref{to}, we may assume $\depth(R)\geq 3$. Following induction hypothesis, $\fp$ is locally free over the punctured spectrum (see Fact \ref{uf}). Let $\mu_r:R\to R$ defined via the assignment $x\mapsto rx$. There is a natural map
$\pi:R\to \Hom_R(\fp,\fp)$ sending $r$ into $\mu_r$. Since $R$ is domain, $\fp$ is torsion-free. Consequently, the map $\pi$ is injective. Then we have $$0\lo R\stackrel{\pi}\lo\Hom_R(\fp,\fp)\lo C:=\coker(\pi)\lo 0\quad(+)$$Since
$\fp$ is locally principal over the  punctured spectrum, $C$ is of finite length. By Grothendieck's vanishing theorem $\hh_{\fm}^{+}(C)=0 $. We plug this in the long exact sequence of local cohomology modules induced by $(+)$ to deduce the following exact sequence
$$0=\hh_{\fm}^{1}(C)\lo\hh_{\fm}^{2}(R)\lo\hh_{\fm}^{2}(\Hom_R(\fp,\fp))\lo\hh_{\fm}^{2}(C)=0 \quad(\dagger)$$
Now, we apply this along with the following embedding: $$\Ext^{1}_R(\fp,\fp)\hookrightarrow\hh_{\fm}^{2}(\Hom_R(\fp,\fp))\stackrel{(\dagger)}=\hh_{\fm}^{2}(R)=0.$$Since $\pd_R(\fp)\leq 1$ it is rigid. Let us apply Jothilingam's result to deduce that $\fp$ is free\footnote{or even without any use of Jothilingam.}.
\end{proof}

\begin{corollary}
Any 3-dimensional regular ring is $\UFD$.
\end{corollary}

\begin{corollary}(Auslander--Buchsbaum)\label{ausb}
	Any  regular ring is $\UFD$.
\end{corollary}

\begin{proof}Recall that the desired property is reduced to the 3-dimensional case. Now, apply the previous corollary.
\end{proof}

\begin{fact}(See \cite[Theorem A]{B}). \label{b2}Let $A$ be a commutative noetherian ring and $M$ a finitely generated $A$-module. Suppose that
	\begin{enumerate}
		\item[i)] $\pd (M)< \infty$,
		\item[ii)] $\End_A(M)$ is a projective $A$-module,
		\item[iii)] $M$ is reflexive.\end{enumerate}
	Then $M$ is a (locally) Gorenstein $A$-module.
\end{fact}
Over normal rings, one has Kaplansky's trick. Let us recover it:
\begin{observation}\label{n}(Kaplansky's trick).
Let $R$ be normal, and $\fp\in\Spec^1(R)$. If $\pd_R(R/\fp)< \infty$, then $\fp$ is principal.
\end{observation}

\begin{proof}
It is easy to see $\Hom_R(\fp,\fp)$ is projective and $\fp$ is reflexive. It remains to apply Fact \ref{b2}.
\end{proof}

Strongly normal means
$(\Se_3)$+$(\R_2)$. Let us extend Kaplansky's trick:
\begin{proposition}\label{trp}	Let $R$ be strongly normal and $\fp\in\Spec^1(R)$. If $\fp$ is tor-rigid, then $\fp$ is principal.
\end{proposition}

\begin{proof}The proof is by induction on $d:=\dim(R)$. Due to $(\R_2)$ condition
	we may assume that $d>2$. In the light of $(\Se_3)$ condition we may and do assume that
	$\depth(R)\geq 3$. By repeating the previous argument    $\Ext^{1}_R(\fp,\fp)\hookrightarrow\hh_{\fm}^{2}(\Hom_R(\fp,\fp))=\hh_{\fm}^{2}(R)=0$. It remains to use tor-rigidity, and conclude that $\fp$ is principal.
\end{proof}
\begin{observation}\label{tori}
	Let $\fp\in\Spec^1(R)$. If $\pd_R(R/\fp)< \infty$. Then $\fp$ is principal iff $\fp$ is tor-rigid.
\end{observation}

\begin{proof}
	It is easy to see.
\end{proof}

\begin{corollary}
	Let $(R,\fm)$ be a local hypersurface ring such that $\widehat{R} = S/(f)$ where $(S, \fn)$ is a
	complete unramified regular local ring and $f$ is a regular element of $S$ contained in $\fn^2$.	Let $\fp\in\Spec^1(R)$. If $\pd_R(R/\fp)< \infty$. Then $\fp$ is principal.
\end{corollary}
\begin{proof}
	Recall from \cite[Theorem 3]{L} every finitely generated  module of finite projective dimension is tor-rigid. Now apply Observation \ref{tori}.
\end{proof}

The notation   $ \id_R(-)$ stands for  injective dimension
of $(-)$. 
\begin{observation}
	Let $R$ be $3$-dimensional, and $\fp\in\Spec^1(R)$. If $\id_R(R/\fp)< \infty$, then $\fp$ is principal.
\end{observation}

\begin{proof}
	By a result of Peskine--Szpiro \cite{PS}, $R$ is Gorenstein. This implies that $\pd_R(R/\fp)< \infty$. As $\dim(R)=3$ and $\depth(R/\fp)>0$, Auslander-Buchsbaum formula says that $\pd_R(R/\fp)< 3$. In view of Proposition \ref{aus}
 $\fp$ is principal.
\end{proof}

\begin{corollary}
	Let $R$ be $(\R_1)$  and $\fp\in\Spec^1(R)$. If $\id_R(R/\fp)< \infty$, then $\fp$ is principal.
\end{corollary}
\begin{proof}Recall that  $R$ is Gorenstein. In particular, $R$ is $(\Se_2)$ and so normal.
	Also, $\pd_R(R/\fp)< \infty$. Now, apply
Observation \ref{n}.
\end{proof}

\begin{remark}
There are strongly normal rings such as $R$ equipped with $\fp\in\Spec^1(R)$ such that $\id_R(\fp)< \infty$. But, $\fp$ is not principal.
\end{remark}
In particular, tor-rigidity assumption is needed in Proposition \ref{trp}.

\begin{proof}
It is enough to consider to the case for which the canonical ideal is a prime ideal. To be more explicit, let $S := \mathbb{C}[[x, y, z, u, v]]$ 
and put $R=S/(yv-zu, yu-xv, xz-y^2).$ It is easy and well-known that
$R$ is 3-dimensional  Cohen-Macaulay, and the Jacobian criterion implies that it is an isolated
singularity. In particular, $R$ is strongly normal. It remains to note that the canonical module is  $(u, v)$.
\end{proof}

\begin{conjecture}
		Let  $\fp\in\Spec^1(R)$. If $\pd_R(R/\fp)< \infty$, then $\fp$ is principal.
	\end{conjecture}
\medskip
\section{A problem by Samuel}
Rings of this section are of zero characteristic.
We are interested in the following problem even we force to put some additional assumptions.
\begin{problem}\label{51}(See  \cite[Last line]{S}).
	Let $R$ be a $d$-dimensional  Cohen-Macaulay ring with isolated singularity and $d>3$. When is $R$  $\UFD$?
\end{problem}

\begin{remark}\label{res}Let us collect a couple of remarks and examples:
\begin{enumerate}
	\item[i)] Let $k$ be a field and $X,Y,Z,W$ be indeterminates.
Let $R$ be the algebra generated by all the monomials involved in $\{X,Y,Z,W\}$ of degree 3. By Gr\"{o}bner, the ring $R$ is Cohen-Macaulay and of dimension four. Since $\Proj(R)\cong \PP^3_k$, $R$ is of graded isolated singularity. But,  $R$ is not $\UFD$. Indeed, suppose it is $\UFD$. Recall that the ring has a canonical module. On the one hand, due to a result of Murphy, see e.g. \cite[Theorem 12.3]{f}, $R$ should be Gorenstein. On the other hand,   $R$ is not Gorenstein, since otherwise $4\equiv_3 0$ which is impossible. In sum, $R$ is not $\UFD$.

	\item[ii)] The divisor class group of a subring  of polynomials is torsion. In particular, the ring $R$ from item i) is almost factorial.

	\item[iii)] The example i) is so special, see \cite[Theorem 1.1]{m}.

	\item[iv)]  Cohen-Macaulay rings with isolated singularity are more general version of Cohen-Macaulay rings of finite Cohen-Macaulay type. Let us ask the problem in this situation. Indeed, it is true for the invariant rings: Let $S = \mathbb{C}[[x_1, \cdots , x_d]]$ and $d>3$. Let $G$ be a finite
group acting faithfully on $S$. Suppose $R:=S^G$ is of finite Cohen-Macaulay type. Then  a result of Auslander--Reiten \cite{ar} implies that the action is trivial. In particular, $R=\mathbb{C}[[x_1, \cdots , x_d]]$ which is $\UFD$ by \ref{ausb}.\end{enumerate}
\end{remark}
 By $\mu(-)$ we mean the minimal number of elements that needs to generate $(-)$.
\begin{proposition}\label{22}
	Let $(R,\fm,k)$ be d-dimensional, Cohen-Macaulay complete containing  $\mathbb{Q}$ and  satisfying $(\R_{d-1})$. If $d>4$ and $\mu(\fm)\leq d+2$, then $R$  is $\UFD$.
\end{proposition}

\begin{proof}
	Suppose $\fm=(x_1,\ldots,x_{d+2})$ and let $A:=k[[X_1,\ldots,X_{d+2}]]$ and denote the natural
	surjection $\pi:A\to R$. Let $\fp:=\ker(\pi)$. Since $R$ is $(\Se_2)$ and $(\R_1)$ it is normal domain, in particular $\fp$ is prime. As $A$ is catenary, $\fp$ is of height two. Now, we are going to use \cite[Theorem 4.9]{ev},  and deduce that $\fp=(a,b)$ for some regular sequence $a,b$. We proved that $R$ is complete-intersection. So, by a result of Grothendieck (see Theorem \ref{groth}), $R$ is $\UFD$.
\end{proof}

\begin{example}
There is a $3$-dimensional  Cohen-Macaulay complete ring $(R,\fm)$ containing  $\mathbb{Q}$ and  satisfying $(\R_2)$. Also, $\mu(\fm)\leq \dim(R)+2$, but $R$  is not $\UFD$.
\end{example}

\begin{proof}
	Let $S := \mathbb{C}[[x, y, z, u, v]]$ 
and put $R=S/(yv-zu, yu-xv, xz-y^2).$
	Recall that
 $R$ is $3$-dimensional  Cohen-Macaulay with isolated
singularity.
It is clear that $\mu(\fm)=\dim(R)+2$. Here we claim that $R$ is not $\UFD$. Indeed, suppose it is $\UFD$. By the mentioned result of Murphy, it should be Gorenstein. But $R$ is not Gorenstein.
\end{proof}
\begin{fact}(Huneke, see \cite[main theorem]{h})\label{fh}
	Let $A$ be a complete local domain containing an infinite field. Suppose
	$A$  is $(\Se_n)$.  If $e(A) < n$,
	then $A$ is Cohen-Macaulay.
\end{fact}

\begin{corollary}
Let $(R,\fm,k)$ be d-dimensional complete normal ring  containing  $\mathbb{Q}$ and  satisfying $(\R_{d-1})$ with $d>4$. If $e(R)\leq 2$ then $R$ is $\UFD$. 
\end{corollary}

\begin{proof} The ring  $R$ satisfies $(\Se _{2})$. According to Fact \ref{fh} $R$ is Cohen-Macaulay.
Following Abhyankar's inequality	\cite{a} we know $\mu(\fm)-\dim R+1 \leq e(R)=2$. Let us combine this along with \cite[Exercise 21.2]{mat}
and observe that $R$ is  complete-intersection\footnote{Huneke \cite[Corollary 4.12]{ev} showed that $R$ is hypersurface provided it contains a field by weaker assumption. Let us apply our elementary approach.}. So, by a result of Grothendieck (see Theorem \ref{groth}), $R$ is $\UFD$. \end{proof}

\begin{corollary}
	Adopt the notation of Problem \ref{51} with $d>4$. If $e(R)\leq 3$ then $R$ is $\UFD$.
\end{corollary}

\begin{proof}
	Following Abhyankar	$\mu(\fm)-\dim R+1 \leq e(R)=3$. In the light of Proposition \ref{22} we get the desired claim.
 \end{proof}

\begin{discussion}\label{mod}
One may reformulate Samuel's problem:
\begin{enumerate}
	\item[i)] Let $R$ be a d-dimensional, Gorenstein ring with isolated singularity. If $d>3$ then is $R$   $\UFD$?

\item[ii)] This is not true: By using Veronese rings, and by the vein such as  Remark \ref{res}, one can find the counter-example, as the divisor class group of a Veronese ring is not trivial, see \cite[16.5]{f}. 

\item[iii)] However,  the problem \ref{mod}(i) is true
if the multiplicity is minimal.

\item[iv)] In addition to Problem \ref{mod}(i) assume $d>7$ and $\mu(\fm)\leq d+3$. Then $R$  is $\UFD$. Indeed, 
suppose $\fm=(x_1,\ldots,x_{d+3})$ and let $A:=k[[X_1,\ldots,X_{d+3}]]$ and denote the natural
surjection $\pi:A\to R$. Look at $\fp:=\ker(\pi)$ which  is prime. As $A$ is catenary, $\fp$ is of height three. Now, we are going to use \cite{HO},  and deduce that $R$ is complete-intersection. So, by a result of Grothendieck (see Theorem \ref{groth}), $R$ is $\UFD$.
\end{enumerate}\end{discussion}
So, the following is clear:
 \begin{proposition}\label{}
	Let $(R,\fm,k)$ be $d$-dimensional, Gorenstein complete  ring containing  $\mathbb{Q}$ and  satisfying $(\R_{d-1})$. If $d>7$ and $\mu(\fm)\leq d+3$, then $R$  is $\UFD$.
\end{proposition}
If $G$ is an abelian group, following Fossum's book there is a Dedekind domain $A$ with $\cl(A)=G$.
 Let $R$ be a Gorenstein ring with isolated singularity. If $\dim R>3$ when is $\cl(R)$ cyclic?

\begin{corollary}
		Let $(R,\fm,k)$ be a complete  ring of depth at least $3$  containing  $\mathbb{Q}$ and  satisfying $(\R_{d-1})$ where $d:=\dim R>4$.  If $e(R)\leq 3$   then $R$ is hypersurface.
\end{corollary}

\begin{proof} The ring satisfies $(\Se _{3})$. According to Fact \ref{fh} $R$ is Cohen-Macaulay.
If $e(R)\le 2$ the claim is well-known by Huneke. So, we may assume that $e(R)=3$.	Following Abhyankar	$\mu(\fm) \leq d+2$. 
	Thanks to Proposition \ref{22}, $R$ is $\UFD$ and following Murphy (see \cite[Theorem 12.3]{f}), $R$ is Gorenstein.
Suppose on the way of contradiction that $R$ is not hypersurface,
then $d+1<\mu(\fm) \leq d+2$. This says that the ring is of 
minimal	multiplicity. Since the residue field is infinite, there is a system of parameter $\underline{x}:=x_1,\ldots,x_d$ so that $\fm^2=(\underline{x}) \fm$. Let us look at the Gorenstein ring  $\overline{R}:=R/(\underline{x})$. It is easy to see $\overline{\fm}^2=0$. In other words, $\Soc(\overline{R})=\overline{\fm}$, which is 1-dimensional. Let $x_{d+1}\in\fm$ be such that $\overline{x}_{d+1}$ is a generator for $\Soc(\overline{R})$. This shows that $\fm:=(\underline{x},x_{d+1})$. 	Let $A:=k[[X_1,\ldots,X_{d+1}]]$ and denote the natural
	surjection $\pi:A\to R$ sending $X_i$ to $x_i$. Let $\fp:=\ker(\pi)$. Since $R$ is $(\Se_2)$ and $(\R_1)$ it is normal domain, in particular $\fp$ is prime. As $A$ is catenary, $\fp$ is of height one. By $\UFD$ property of regular rings there is  some  $f$ such that $\fp=(f)$. We proved that $R=A/(f)$, i.e., it is hypersurface. This contradiction completes the proof. 
\end{proof}\medskip
\section{Relative situations}

We start by the following recent result conjectured by Gabber:

\begin{fact}(\v{C}esnavi\v{c}ius-Scholze, see \cite[Theorem 1.1.3]{cep}).\label{pe}
Let $(R,\fm)$ be a  complete
intersection.
If $\dim (R)\geq 3$, then $\Pic(\Spec(R)\setminus\{\fm\})_{\tors}=0$.
\end{fact}
A \emph{quasi-deformation} of $R$ is a diagram $R\rightarrow A\twoheadleftarrow Q$ of local homomorphisms, in which $R\rightarrow A$ is faithfully flat, and $A\twoheadleftarrow Q$ is surjective with kernel generated by a regular sequence. The \emph{complete intersection dimension} of $M$, see   \cite{AGP}, is:
$$\CI_R(M)=\inf\{\pd_Q(M\otimes_RA)-\pd_Q(A)\mid R\rightarrow A\twoheadleftarrow Q \text{ is a quasi-deformation}\}.$$

\begin{theorem}
Let $(R,\fm)$ be a strongly normal almost factorial complete-intersection ring and 	$\fp\in\Spec^1(R)$. If $\depth(R/\fp)\geq\dim R-2$, then $\fp$ is principal.
\end{theorem}

\begin{proof}
The proof is proceed by induction on
$d:=\dim(R)$. Suppose first
that
$d<4$. Due to the $(R_2)$ condition, we may adopt the situation of Fact \ref{pe}.
 Recall in this case that  $\Pic(\Spec(R)\setminus\{\fm\})\stackrel{\cong}\lo \cl(R)$ as the ring is with isolated singularity (see \cite[Proposition 18.10(b)]{f}). Thanks to the almost factorial assumption, $$\cl(R)=\cl(R)_{\tors}=\Pic\big(\Spec(R)\setminus\{\fm\}\big)_{\tors}=0.$$ Since the ring is normal we deduce that $R$ is $\UFD$, and in particular $\fp$ is principal.
Now, assume $d>3$. One may reformulate 
$\depth(R/\fp)\geq\dim R-2$ by $\CI(R/\fp)\leq 2$.
Since  $\CI(S^{-1}M)\leq \CI(M)$ (see \cite[1.6]{AGP}) we may assume that $\depth(R_Q/\fp R_Q)\geq \dim( R_Q)-2$ for all prime ideal $Q$ such that $Q\in\V( \fp)$. In the light of \cite[Cor 7.2]{f} we know  $\cl(R)\twoheadrightarrow \cl(S^{-1}R)$ is surjective. So, $\cl(S^{-1}R)$ is torsion, i.e., $S^{-1}R$ is almost factorial.
Due to
inductive hypothesis, we deduce that $\fp$ is locally principal over $\Spec(R)\setminus\{\fm\}$.
Recall that $d>3$ and $\depth(R/\fp)\geq\dim R-2\geq 2$. From this and $0\to \fp\to R\to R/ \fp \to 0$ we deduce that 
$\depth(\fp)\geq 3$. Recall the ring is normal
and of depth at least four. Now, we apply Lemma \ref{t1}
$$\Ext^2_R(\fp,\fp)\stackrel{}\hookrightarrow \hh^3_{\fm}(\Hom_R(\fp,\fp))\cong\hh^3_{\fm}(R)=0.$$
We combine this along with 	\cite{AvBu} and deduce  that $\pd_R( \fp)$ is finite. This means that $\pd_R(R/\fp)=\CI(R/\fp)\leq 2$.
It remains to apply Proposition \ref{aus}, and deduce that 
$\fp$ is principal\footnote{or even without any use of Proposition \ref{aus}.}.
\end{proof}

\begin{conjecture}Suppose  $\dim (R)\geq 4$ with isolated singularity, and let  $\fp\in\Spec^1(R)$ be such that $\CI_R(\fp)< \infty$. Then $\fp$ is principal.
\end{conjecture}

\begin{discussion}
	An $R$-module $M$ is called \emph{totally reflexive} provided that: 	\begin{enumerate}
		\item[i)] the natural map $M\rightarrow M^{**}$ is an isomorphism,
		\item[ii)] $\Ext^i_R(M,R)=\Ext^i_R(M^*,R)=0$ for all $i\geq 1$.
	\end{enumerate}  The \emph{Gorenstein dimension} of $M$, denoted $\Gdim_R(M)$, is defined to be the infimum of all nonnegative integers $n$, such that there exists an exact sequence  $0\rightarrow G_n\rightarrow\cdots\rightarrow G_0\rightarrow  M \rightarrow 0,$  in which each $G_i$ is a totally reflexive $R$-module. 
	Every finitely generated module over a Gorenstein ring has finite Gorenstein dimension. Moreover, if $R$ is local and $\Gdim_R(M)<\infty$, then it follows that $\Gdim_R(M)=\depth R-\depth_R(M)$, and we call it Auslander-Bridger formula. If  $\Gdim_R(k)<\infty$ then $R$ is  Gorenstein.
	For more details, see  \cite{AB2}.
\end{discussion}

\begin{example}
	Let $R$ be a $3$-dimensional Gorenstein ring which is not $\UFD$, e.g. $R=k[[x,y,u,v]]/(u^2)$. Let $\fp\in\Spec^1(R)$ be  non-principal, e.g., $\fp:=(u,v)$. Then $\Gdim(R/\fp)\leq 2$ but $\fp$ is not free.
\end{example}
\begin{question}\label{fgr}
	Let $\fp\in\Spec^1(R)$ be such that $\Gdim(R/\fp)\leq 2$. When is $\fp$  totally reflexive?
\end{question}
\begin{lemma}\label{tog}
	Let $R$ be of depth at most two, and $\fp\in\Spec^1(R)$. If $\Gdim(R/\fp)< \infty$, then $\fp$ is totally reflexive.
\end{lemma}

\begin{proof}
	Suppose first that $\fp=\fm$. 	Then
	$\Gdim_R(k)<\infty$.
	This shows $R$ is  Gorenstein, and $1$-dimensional as $\fm$ is of height one.
	In this case any ideal is totally reflexive.
	Now,  assume $\fp\neq\fm$. Then $\depth(R/\fp)>0$ as $\fp$ is prime. By Auslander-Bridger formula,
	$\Gdim(R/\fp)\leq1$ because $R$ is of depth at most two.  Thanks to $0\to \fp\to R\to R/\fp\to 0$ we observe $\fp$ is totally reflexive.
\end{proof}

\begin{proposition}
	Let	$R$ be $(\Se_3)$ and $\fp\in\Spec^1(R)$ be  such that $\fp^\ast$ is $(\Se_3)$. If $\Gdim(R/\fp)\leq 2$, then	$\fp$ is totally reflexive.
\end{proposition}

\begin{proof} 
	The proof is by induction on $d:=\dim R$. The case $d<3$ is in Lemma \ref{tog}. So, we may assume $d>2$. Consequently, $\depth(R)\geq 3$.  Recall that
	$\Gdim(R_Q/\fp R_Q)\leq 2$	for all $Q\in\V(\fp)\setminus\{\fm\}$.
	Thanks to inductive hypothesis
	$\fp R_Q$ is totally reflexive for all $Q\in\Spec(R)\setminus\{\fm\}$. From this 
	$E:=\Ext^1_R(\fp,R)$ is of finite length.
	If a module $(-)$ has finite $\Gdim$, then
	$\Gdim(-)=\sup\{i:\Ext^i_R(-,R)\neq 0\}$. So, it is enough to show $E=0$ . There is a free module $F$ and a totally reflexive module $T$ such that the sequence
	$ 0\to T\to F\to \fp\to 0$  is exact. This gives
	$0\to\fp^\ast\to F^\ast\to T^\ast\to E\to 0$. Let us breakdown it into two short exact sequences\begin{enumerate}
		\item[i)]
		$0\to\fp^\ast\to F^\ast\to S\to 0$,
		\item[ii)] 
		$0\to S\to T^\ast\to E\to 0$.
	\end{enumerate}
	$ii)$ yields that  $0=\hh_{\fm}^{0}(T^\ast)\to\hh_{\fm}^{0}(E)\to\hh_{\fm}^{1}(S)\to\hh_{\fm}^{1}(T^\ast)=0, $ and so  $E=\hh_{\fm}^{0}(E)\cong\hh_{\fm}^{1}(S)$. Since
	$\depth(R)>2$ from i) we get that $\hh_{\fm}^{1}(S)\cong\hh_{\fm}^{2}(\fp^\ast)$. Combining these, imply that $E\cong\hh_{\fm}^{2}(\fp^\ast)=0$.
\end{proof}
The above argument shows the following:
\begin{corollary}
	Let	$R$ be $3$-dimensional Gorenstein and $\fp\in\Spec^1(R)$. Then	$\fp^\ast$ is generalized Cohen-Macaulay.
\end{corollary}

\begin{corollary}
	Let	$R$ be $3$-dimensional normal Gorenstein and $\fp\in\Spec^1(R)$. Then	$\fp$ is generalized Cohen-Macaulay.
\end{corollary}

\begin{proof}The normality condition implies $\fp$ is reflexive. Thanks to the previous corollary, $\fp^\ast$ is generalized Cohen-Macaulay.
In the light of \cite[Corollary 4.4]{weak} we observe that  $\fp=(\fp^\ast)^\ast$ is generalized Cohen-Macaulay.
\end{proof}
In the book \cite{ev} Kaplansky's trick (see Observation \ref{n}) drives from the syzygy theorem of Evans and Griffith plus the direct summand conjecture. This suggests to
find the Gorenstein analogue of syzygy theorem. The natural candidate is:

\begin{question}
	Let $R$ be a normal local ring and which satisfies $(\Se_k)$. Let $M$ be $k$-th
	syzygy and of finite $\G$-dimension. Suppose $M$ is not totally reflexive. When is $\rank(M)\geq k$?
\end{question}

Over normal rings, positive answer to this gives an affirmative solution to Question \ref{fgr}.

\begin{question}
	Let $\fp\in\Spec(R)$ be such that $\Gdim(R/\fp)< \infty$. When is $R$ generically Gorenstein?
\end{question}
Generically Gorenstein is the $(G_0)$ condition.
\medskip

\section{A question by Braun}
The rings in this section are equipped with a kanonical module, for example homomorphic image of Gorenstein rings are of this mood.

 \begin{question} (Braun, \cite[Question 16]{B}). \label{b}Let $(R,\fm)$ be a  normal  domain and $I\lhd R$ a reflexive ideal with $\id_R( I) <
	\infty$. Is $I$ isomorphic to a canonical module?
\end{question}

By \cite[Page 682]{B}, the only positive evidence we have is when $R$ is also Gorenstein. The $2$-dimensional case answered in \cite{moh}:

\begin{lemma} \label{7.2} Let $(R,\fm)$ be a  normal  domain  of dimension $2$ with a canonical module and $I\lhd R$ be reflexive with $\id_R( I) <
	\infty$. Then $I$ isomorphic to a canonical module.
	\end{lemma}

\begin{fact}(See \cite[Theorem C]{B}). \label{b1}Let $A$ be a commutative Noetherian ring and $M$ a finitely generated $A$-module. Suppose that
\begin{enumerate}
	\item[i)] $\id (M)< \infty$,
\item[ii)] $\End_A(M)$ is a projective $A$-module,
\item[iii)] $\Ext^1_A(M,M) = 0$.\end{enumerate}
	Then $M$ is a (locally) Gorenstein $A$-module.
	\end{fact}

Now, we are ready to prove:
\begin{theorem}
 Question \ref{b} is true.
\end{theorem}

\begin{proof}The proof proceeds by induction on
	$d:=\dim(R)$. Thanks to Lemma \ref{7.2}
	we may assume that $d=\dim R>2$ and suppose the desired claim is satisfied for normal rings of dimensions less than $d$.
Now, let $I\lhd R$ be reflexive with $\id_R( I) <
\infty$. Following Bass' conjecture $R$ is Cohen-Macaulay. In other words,
we may assume that $$d=\dim R=\depth(R)>2\quad(+).$$Recall that localization  of $I$
is a divisorial reflexive ideal and of finite injective dimension.
In particular, by 
applying the inductive hypothesis, we deduce that $I$ is locally isomorphic with the canonical module over $\Spec(R)\setminus\{\fm\}$.
Recall that
$$\Ext^+_{R}(I,I)_Q=\Ext^+_{R_Q}(I_Q,I_Q)=\Ext^+_{R_Q}(\omega_{R_Q},\omega_{R_Q})=0$$	for all $Q\in\V(I)\setminus\{\fm\}$. Also, $\Ext^+_{R}(I,I)_Q=0$ if $Q\notin\V(I)$ because 
$I_Q=R_Q.$ From this 
$\Ext^+_R(I,I)$ is of finite length. Since the ideal is reflexive, we know $\depth(I)\geq2$. This allows us to apply Lemma
\ref{t1} with $t:=2$.
Recall that $R$ is normal.  Applying this along with the determinate trick, we observe that 
$$R\subseteq\Hom_R(I,I)\subseteq\overline{R}=R\quad(\ast)$$By Lemma
\ref{t1} we know there is an injection $$\Ext^{1}_R(I,I)\hookrightarrow\hh_{\fm}^{2}(\Hom_R(I,I))\stackrel{(\ast)}=\hh_{\fm}^{2}(R)\stackrel{(+)}=0.$$  Consequently, 
$\Ext^{1}_R(I,I)=0$. In the light of  Fact \ref{b1} we observe that  $I $ is a  Gorenstein module, and so Cohen-Macaulay. But, $I$ is of full support.  Since  $ I$ is both maximal Cohen-Macaulay and
of finite injective dimension, we know that $I\simeq\oplus_n\omega_R$ for some $n$. Recall that $\Hom_R(\omega_R,\omega_R)=R$. This shows $$R\stackrel{(\ast)}\cong\Hom_R(I,I)\cong
\Hom_R(\oplus_n\omega_R,\oplus_n\omega_R)\cong\oplus_{n^2}\Hom_R(\omega_R,\omega_R)\cong\oplus_{n^2} R,$$i.e., $n^2=1$. Consequently, $I$ is isomorphic to the canonical module.
\end{proof}
\medskip
\section{Reflexive hull and ideal transformation}

By ideal transformation we mean
$D_{\fm}(M) := \bigcup_{n=1}^{\infty} \Hom(\mathfrak{m}^n, M)$.

\begin{fact}[After Hartshorne]\label{ah}
	Let \( M \) be locally free over \( \operatorname{Spec}(R) \setminus \{\mathfrak{m}\} \) and let \( R \) be a local ring with \( \operatorname{depth}(R) \ge 2 \). There is an exact sequence
	\[
	0 \lo \Gamma_{\fm}(M) \lo M \lo M^{**} \lo H^1_{\fm}(M) \lo 0 .
	\]
	In particular,
	\(
	M^{**} = D_{\fm}(M).
	\)
\end{fact}

\begin{proof}Let $f:M\to M^{**} $ be the natural map.
	It is well known to almost everyone that
	\begin{itemize}
		\item \( \ker(f) = \Ext^1(\Tr M, R), \)
		\item \( \coker(f) = \Ext^2(\Tr M, R), \)
	\end{itemize}
with the convention that $\Tr(-)$ denotes Auslander's transpose.
	Also, it is well known    that (see e.g. \cite[Lemma 3.5(i)]{ACS}):
	\[
	\Ext^1(\Tr M, R) = H^0_{\mathfrak{m}}(M), \qquad 
	\Ext^2(\Tr M, R) = H^1_{\mathfrak{m}}(M).
	\]

	Thus, there is a commutative diagram
	\[
	\begin{array}{ccccccccccc}
	0 & \longrightarrow{} & \Gamma_{\fm}(M) & \lo{} & M & \lo{} & M^{**} & \lo{} & H_\fm^1(M) & \lo{} & 0 \\
	& & \downarrow{=} &    &\downarrow{=} & & \downarrow{\exists\beta} &  &  \downarrow{=} & & \\
	0 & \lo{} & \Gamma_{\fm}(M) & \lo{} & M & \lo{} & D_\fm(M) & \lo{} & H_\fm^1(M) & \lo{} & 0
	\end{array}
	\]
	In particular,  by the 5-lemma \( \beta \) is an isomorphism, i.e., \( M^{**} \cong D_\fm(M) \).
\end{proof}

\begin{remark}
	The assumption \(\operatorname{depth}(R) \ge 2\), in particular \(\dim R \ge 2\) when \(R\) is Cohen--Macaulay, is essential. 
If \(\dim R = 1\), then \(H^1_{\fm}(M)\) is not finitely generated as soon as $\dim (M)>0$, and consequently \(D_{\fm}(M)\) is not finitely generated. However, the reflexive hull \(M^{**}\) is always finitely generated. Hence the equality \(M^{**} = D_{\fm}(M)\) fails when \(\dim R = 1\).
\end{remark}
\begin{corollary}
	Let \( M \) be locally free over \( \operatorname{Spec}(R) \setminus \{\mathfrak{m}\} \) and let \( R \) be a local ring with \( \operatorname{depth}(N) \ge 2 \). Then
	\(
 D_{\fm}(M\otimes_RN)\cong \Hom_R(M^{*},N).
	\)
\end{corollary}

\begin{proof}
By applying the previous argument along with the following sequence \[
0 \lo \Gamma_{\fm}(M\otimes N) \lo M\otimes N \lo \Hom_R(M^{*},N) \lo H^1_{\fm}(M\otimes N) \lo 0 ,
\]we get the claim.
\end{proof}
Set \(
(-)^\vee := \Hom(-, E_R(k)).
\)
\begin{corollary}
	Assume in addition to the previous result that $R$ is Cohen-Macaulay with canonical module. Then  $\Ext^1_R(\Tr M, R) =  \Ext^d_R(M, \omega_R)^\vee$ and $\Ext^2_R(\Tr M, R) =  \Ext^{d-1}_R(M, \omega_R)^\vee$.
\end{corollary}

\begin{proof}
	This follows by an application of local duality in connection with {Fact} \ref{ah}.
\end{proof}

\begin{acknowledgement}
	I thank   Olgur Celikbas for useful comments on the earlier draft. 
\end{acknowledgement}

\end{document}